\documentclass{amsart}

\usepackage{amsmath}
\usepackage{graphicx}
\usepackage{float}
\usepackage{xcolor}
\usepackage{url}

\newtheorem{theorem}{Theorem}[section]
\newtheorem{lemma}[theorem]{Lemma}
\newtheorem{corollary}[theorem]{Corollary}
\newtheorem*{thmA}{Theorem A}
\newtheorem*{thmB}{Theorem B}

\theoremstyle{definition}
\newtheorem{definition}[theorem]{Definition}
\newtheorem{que}{Question}

\theoremstyle{remark}
\newtheorem{remark}[theorem]{Remark}

\DeclareMathOperator{\diam}{diam}
\DeclareMathOperator{\gend}{end}
\DeclareMathOperator{\mesh}{mesh}

\begin{document}

\title{Cantor subsystems on the Gehman dendrite}



\author{Piotr Oprocha}
\address{AGH University of Krakow, Faculty of Applied Mathematics, al. Adama Mickiewicza~30, 30-059 Krak\'ow, Poland \&
Centre of Excellence
IT4Innovations - Institute for Research and Applications of Fuzzy Modeling,
University of Ostrava, 30. Dubna 22, 701 03 Ostrava 1, Czech Republic
}

\email{oprocha@agh.edu.pl}

\author{Jakub Tomaszewski}
\address{AGH University of Krakow, Faculty of Applied Mathematics, al. Adama Mickiewicza~30, 30-059 Krak\'ow, Poland
}

\email{tomaszew@agh.edu.pl}


\date{}

\dedicatory{}


\begin{abstract}
In the present note we focus on dynamics on the Gehman dendrite $\mathcal{G}$. It is well-known that the set of its endpoints is homeomorphic to a standard Cantor ternary set. 
For any given surjective Cantor system $\mathcal{C}$ we provide constructions of (i) a mixing but not exact and (ii) an exact map on $\mathcal{G}$, such that in both cases the subsystem formed by $\text{End}(\mathcal{G})$ is conjugate to the initially chosen system on $\mathcal{C}$.
\end{abstract}

\maketitle

\section{Introduction}
The dynamics on the Gehman dendrite $\mathcal{G}$ is relatively well understood and leads to interesting examples.
The authors of \cite{Kocan} provided an example of a map $F\colon \mathcal{G}\to\mathcal{G}$ such that the subsystem $(\text{End}(\mathcal{G}), F|_{\text{End}(\mathcal{G})})$ is conjugate to a full shift on two symbols. By the well known characterization of Gehman dendrite \cite{Charatonik} it immediately leads to examples of map $F$ with any infinite subshift without isolated points conjugate to $F|_{\text{End}(\mathcal{G})}$, e.g. see \cite{Li}. These maps are not transitive, which leads to the question that motivates this paper:
\begin{que}\label{q1}
Which Cantor dynamical systems can arise as invariant sets of transitive maps on the Gehman dendrite?
\end{que}

Before providing an answer, let us recall a few related questions. In \cite{Bald01}, Baldwin asked if the infimum of entropies of transitive maps on any dendrite which is not a tree is zero (it is known that it is positive on trees) and if zero is attainable by a transitive map on any dendrite. He also developed a generalization of kneading theory, to describe dynamics on some dendrites \cite{Bald07}. The question of Baldwin was partially answered by \v{S}pitalsk\'y in \cite{Spi}, who proved that if $D$ is a non-degenerate dendrite such that no subtree of it contains all free arcs of $D$ then $D$ admits a transitive (or even exact) map with arbitrarily small entropy. Clearly, the Gehman dendrite is a special case of dendrites covered by this result. It is also worth recalling another result by \v{S}pitalsk\'y \cite{Spi2}, which shows that on Gehman dendrite transitive maps always have positive entropy.

The above considerations somewhat touch the surface of a more general question: what different types of chaos (and how many of them at the same time) can we observe in a dynamical system? There are many classical variations of chaos we may consider, e.g. Devaney chaos, Auslander-Yorke chaos, but also Li-Yorke chaos, $\omega$-chaos, distributional chaos and many more. By studying invariant subsystems in the dendrite case, we shed a light onto a subject of admittable chaotic behavior in the local sense. What's more, we take our considerations even further, setting the goal to look for global chaotic properties as well. 

The Cantor systems has been one of the key research interests in dynamics over the recent years. Notable advances in this topic include e.g. a study of generic maps in the space of homeomorphisms $\mathcal{H}(\mathcal{C})$. It has been shown by Glasner and Weiss in \cite{Glasner} that this space contains a dense conjugacy class. This result has been presented also in \cite{monografia}, in which the authors asked whether $\mathcal{H}(\mathcal{C})$ contains a comeager conjugacy class. A positive answer to that question has been given by Kechris and Rosenthal in \cite{Kechris}
and later extended by many others, including \cite{Akin, Darji} or \cite{Kupka} to name a few. These result show that in fact there is one typical map on Cantor set which is far from being chaotic of any type. Fortunately, its complement contains a plethora of admissible dynamics.
From that point of view especially significant is  the theory of symbolic extensions for dynamical systems, a topic that arose from a statement included in \cite{Weiss}: every homeomorphism of a compact metric space admits a zero-dimensional extension with the same topological entropy. This theory has been thoroughly studied by Downarowicz et al. over the last two decades, see e.g. \cite{Boyle,Downarowicz}.

Let us now turn to answer to Question~\ref{q1}.
The key ingredient is to properly define the action on the branching points of $\mathcal{G}$ and then extend the action onto edges of $\mathcal{G}$. The natural extension of $F$ onto the set $\text{End}(\mathcal{G})$ will result with the subsystem $(\text{End}(\mathcal{G}), F|_{\text{End}(\mathcal{G})})$ conjugate to the Cantor system of choice. The control in the construction will be provided by graph covers from \cite{Shimomura}, which represent dynamics on the Cantor set $\mathcal{C}$ as an inverse limit of a special kind. This leads to the following.

\begin{thmA}\label{thm:A}
Let $(\mathcal{C}, f)$ be a Cantor dynamical system. Then there exists a pure mixing dynamical system $(\mathcal{G}, F)$ such that the subsystem $(\text{End}(\mathcal{G}), F|_{\text{End}(\mathcal{G})})$ is conjugate to $(\mathcal{C}, f)$.
\end{thmA}

The resulting map exhibits at least one global type of chaos, in the Devaney sense, as well as an arbitrary variety of local chaotic behaviors one may choose together with the initial Cantor system to extend.

The construction provided in the proof of Theorem A can be modified to obtain an exact map on the Gehman dendrite while preserving the conjugacy of the endpoint subsystem. It is based on a result from \cite{Nadler} (see Theorem~\ref{GMT}) which utility in construction of exact maps we learned from Karasová and Vejnar (see e.g. \cite{Klara} for their application of this method). This modification forces to improve our observable chaos list: we no longer have just Devaney chaos, but the more complex exact Devaney chaos (see \cite{exact_devaney} for more details).

\begin{thmB}\label{thm:B}
Let $(\mathcal{C}, f)$ be a Cantor dynamical system. Then there exists an exact dynamical system $(\mathcal{G}, F)$ such that the subsystem $(\text{End}(\mathcal{G}), F|_{\text{End}(\mathcal{G})})$ is conjugate to $(\mathcal{C}, f)$.
\end{thmB}

 In our constructions we use large stretching and multiple overlapping of images of arcs, which suggests that the entropy of constructed maps can be arbitrarily large and thus allows for yet another, entropy indicated, chaos to be observed in the resulting dynamical system. This motivates the following question which we leave for further research.

\begin{que}
Let $(\mathcal{C}, f)$ be a surjective Cantor system and let $\varepsilon>0$.
Does there exists a transitive dynamical system $(\mathcal{G}, F)$ such that the subsystem $(\text{End}(\mathcal{G}), F|_{\text{End}(\mathcal{G})})$ is conjugate to $(\mathcal{C}, f)$ and topological entropy satisfies $h_\text{top}(f)>h_\text{top}(F)-\varepsilon$?
\end{que}
In other words, can any Cantor dynamical system be extended to transitive dynamical system of the Gehman dendrite without increasing entropy more than by $\varepsilon$?

\section{Preliminaries}
\par Let $(X,f)$ be a topological dynamical system (a TDS for short), where $X$ is assumed to be a compact metric space and $f\colon X\to X$ is a continuous surjection. For such TDS we define the following properties, characterizing its local behavior over time:

\begin{definition}
We say that a TDS $(X,f)$ is transitive if for each pair of nonempty open subsets $U,V\subset X$ there exists $n\in\mathbb{N}$ such that $f^{-n}(U)\cap V\neq\emptyset$.
\end{definition}

\begin{definition}
    We say that a TDS $(X,f)$ is topologically mixing if for every nonempty open sets $U,V\subset X$ there exists $n_0\in\mathbb{N}$ such that for all $n\geq n_0$ we have $f^{-n}(U)\cap V\neq\emptyset$.
\end{definition}

\begin{definition}
    We say that a TDS $(X,f)$ is exact if for each nonempty open set $U\subset X$ there exists $n\in\mathbb{N}$ such that $f^{n}(U)=X$.
\end{definition}
We often say that a map $f\colon X\to X$ is transitive (respectively mixing, exact) and by that we mean that the TDS $(X,f)$  has that property. Given the space $X$, there exist maps on $X$ which are mixing but not exact. We call those maps pure mixing.

\begin{definition}
    Let $(X,f)$ and $(Y,g)$ be two TDS. We say $(Y,g)$ is semiconjugate to $(X,f)$ if there exists a continuous surjection $\phi\colon X\to Y$ such that $\phi\circ f=g\circ\phi$. Moreover, if $\phi$ is a homeomorphism, then we say that $(X,f)$ and $(Y,g)$ are conjugate.
\end{definition}

\subsection{Graph covers for Cantor systems}\label{approximation}
In this section we will introduce the necessary notions to approximate Cantor dynamics. Our presentation follows \cite{Shimomura}.
\par Let $G=(V,E)$ and $G'=(V', E')$ be graphs. A map $\phi\colon V\to V'$ is a graph homomorphism if $(u,v)\in E\implies (\phi(u), \phi(v))\in E'$. Adopting the notation from \cite{Shimomura}, whenever this is the case, we write $\phi\colon G\to G'$ or $\phi\colon (V,E)\to (V', E')$.
\begin{definition}
    Let $\phi\colon G\to G'$ be a graph homomorphism. We say that $\phi$ is vertex-surjective if $V'=\phi(V)$. Similarily, we say that $\phi$ is edge-surjective if $E'=(\phi\times\phi)(E)$.
\end{definition}
\begin{definition}
    Let $\phi\colon (V,E)\to (V', E')$ be a graph homomorphism. Assume that the following condition is satisfied:
    \begin{equation*}
        (u,v), (u,v')\in E\implies \phi(v)=\phi(v').
    \end{equation*}
    Then we say that $\phi$ is + directional.
\end{definition}
\begin{definition}
    A graph homomorphism $\phi\colon G\to G'$ is called a cover if it is a + directional edge-surjective graph homomorphism.
\end{definition}
Note that, following \cite{Shimomura}, we will use the notion of a "graph cover" and "cover" interchangeably, assuming that these notions signify the same mathematical object.
\begin{definition}
    We say that a family $\mathcal{U}$ of open subsets of the space $X$ is an open cover if $\bigcup_{U\in\mathcal{U}}U =X$. Moreover, we say that such a family is a decomposition if for each two subsets $U,V\in\mathcal{U}$ we have $U\cap V=\emptyset$. 
\end{definition}
Yet again, for Cantor systems, we will use these notions interchangeably. Let $$G_1\xleftarrow[]{\phi_1}G_2\xleftarrow{\phi_2}G_3\xleftarrow[]{\phi_3}...$$ be a sequence of graph covers, to which we attach a singleton graph $G_0=(\overline{0}, (\overline{0},\overline{0}))$ to the head with natural graph cover $\phi_0$. On the inverse space
\begin{equation*}
    \mathrm{G}=G_0\xleftarrow[]{\phi_0}G_1\xleftarrow[]{\phi_1}G_2\xleftarrow{\phi_2}G_3\xleftarrow[]{\phi_3}...
\end{equation*}
we define the inverse limit
\begin{equation*}
    V_{\mathrm{G}}=\left\{
    (x_0,x_1, x_2,...)\in \prod_{i=0}^\infty V_i\colon\;x_i=\phi_i(x_{i+1})\text{ for }i\in\mathbb{N}
    \right\}
\end{equation*}
equipped with product topology and the relation
\begin{equation*}
    E_{\mathrm{G}}=\left\{
    (x,y)\in V_{\mathrm{G}}\times V_{\mathrm{G}}\colon\; (x_i,y_i)\in E_i\text{ for } i\in\mathbb{N}
    \right\}
\end{equation*}

\begin{theorem}[\cite{Shimomura}, Lemma 3.5]
Let $\mathrm{G}$ be a sequence of covers. Then $V_{\mathrm{G}}$ is a compact metrizable 0-dimensional space and $E_{\mathrm{G}}$ determines a continuous map from $V_{\mathrm{G}}$ to itself.
\end{theorem}
Assume that $(X,f)$ is a dynamical system and let $\mathcal{U}=(U_i)_{i=1}^n$ be an open cover of $X$. On $\mathcal{U}$ we define the following relation (denoted by $f^{\mathcal{U}}$)
\begin{equation*}
    U_i\sim U_j \iff f(U_i)\cap U_j\neq \emptyset.
\end{equation*}
One can see that the pair $(\mathcal{U}, f^{\mathcal{U}})$ forms a graph. Moreover, as one can find in \cite{Shimomura}, if $(\mathcal{C}, f)$ is a surjective Cantor system, then one can construct a refining sequence of decompositions $(\mathcal{U}_i)$ such that the dynamical system $(V_{(\mathcal{U}_i)}, E_{(\mathcal{U}_i)})$ represents the dynamics on $(\mathcal{C}, f)$.
\begin{theorem}[\cite{Shimomura}, Lemma 3.8]\label{shimomura_tower}
    Let $(\mathcal{C}, f)$ be a surjective Cantor system. Then there exists a refining sequence of decompositions $(\mathcal{U}_i)$ such that
    \begin{enumerate}
        \item the sequence $\mathrm{U}=(\mathcal{U}_1, f^{\mathcal{U}_1})\xleftarrow[]{\phi_1}(\mathcal{U}_2, f^{\mathcal{U}_2})\xleftarrow[]{\phi_1}...$ of natural graph homomorphisms is a sequence of covers,
        \item the system $(V_\mathrm{U}, E_\mathrm{U})$ is conjugate to $(X,f)$.
    \end{enumerate}
\end{theorem}
\begin{definition}
     Let $\mathcal{U}$ be a collection of subsets of a metric space $X$. We define the mesh of $\mathcal{U}$ to be the number $\mesh(\mathcal{U})=\sup\{\diam(U)\colon U\in\mathcal{U}\}$.
\end{definition}
\begin{corollary}[of the proof of \cite{Shimomura}, Lemma 3.8]\label{shimomura_fix}
Let $(\mathcal{C}, f)$ be a surjective Cantor system and $\mathcal{U}_0$ a decomposition of $\mathcal{C}$. Set
\begin{equation*}
    \mathrm{U}=(\mathcal{C}, (\mathcal{C},\mathcal{C}))\xleftarrow[]{\phi_0}(\mathcal{U}_1, f^{\mathcal{U}_1})\xleftarrow[]{\phi_1}(\mathcal{U}_2, f^{\mathcal{U}_2})\xleftarrow[]{\phi_2}...,
\end{equation*}
where each graph homomorphism $\phi_i$ is induced by the identity map on $\mathcal{C}$ and each decomposition $\mathcal{U}_{i+1}$ is a refinement of $$f^{-1}(\mathcal{U}_i)\vee \mathcal{U}_i=\{U\cap V\colon\;U\in f^{-1}(\mathcal{U}_i),\;V\in\mathcal{U}_i\}$$
 with $\mesh(\mathcal{U}_i)\to 0$. Then the system $(V_\mathrm{U}, E_\mathrm{U})$ is conjugate to $(X,f)$.
\end{corollary}

\section{Proof of the Theorem A}
Let $(\mathcal{C}, f)$ be an initially chosen Cantor dynamical system. In what follows, we will explain how the structure of $\mathcal{C}$ can be used to define a Gehman dendrite with "marked" branching points and associate with this structure a mixing map $F$ with desired properties.


\subsection{Setup}\label{sec:setup}
Let $\mathcal{G}$ be a Gehman dendrite, that is the topologically unique dendrite  (up to homeomorphism) whose set of endpoints is homeomorphic to the Cantor set and whose branch points are all of  Menger-Urysohn order 3 (see \cite{Charatonik, Nadler, Nikiel}  for further details). We know that, with due diligence, $\mathcal{G}$ can be constructed as a closure of a union of countably many arcs in $\mathbb{R}^2$. Let us shortly present the main idea together with a suitable notation for our further considerations. We begin with an initial point $c\in\mathbb{R}^2$, which we will call the root of $\mathcal{G}$. We attach two arcs $A_0 = [c, c_0]$, $A_1 = [c, c_1]$ in such a way that $(c, c_0], (c, c_1]$ are disjoint. Then we continue in the same way, attaching two arcs to each of the other endpoints of arcs attached in the previous step. We set $A_{00}=[c_0, c_{00}], A_{01}=[c_0, c_{01}], A_{10}=[c_1, c_{10}]$ and $A_{11}=[c_1, c_{11}]$. Then for indexing sequence at step $n\geq 2$, $\omega=(\omega_0,\ldots,\omega_{n-1})\in \{0,1\}^n$ and a corresponding arc $A_\omega=[c_{\omega{[0,n-2]}},c_{\omega}]$, where $\omega{[0,n-2]}=(\omega_0,...,\omega_{n-2})$ is the restriction of $\omega$ to the first $n-1$ digits, we attach two arcs $A_{\omega0}=[c_\omega,c_{\omega0}]$ and $A_{\omega1}=[c_\omega,c_{\omega1}]$. The indexing sequence $\omega=(\omega_0,\ldots, \omega_{n-1})$ uniquely characterizes the arc in $\mathcal{G}$ with one endpoint at $c_{\omega{[0,n-2]}}$ and second at $c_{\omega}$.

\begin{definition}
    A path $P=[c_\omega, c_{\omega\alpha}]$ for some extension $\alpha\in\{0,1\}^n$ of the binary code $\omega$ is a connected union $$[c_\omega, c_{\omega\alpha}]=[c_\omega, c_{\omega\alpha_1}]\cup...\cup[c_{\omega\alpha_1...\alpha_{n-1}}, c_{\omega\alpha}].$$
\end{definition}

All arcs indexed with sequences of length $n$ will be called level $n$. Moreover, the construction by attaching consecutive arcs and the positioning of those arcs with respect to the root of $\mathcal{G}$ provides a natural orientation of the dendrite. Thus any binary sequence characterising an arc $[c_v,c_u]\subset\mathcal{G}$ characterizes also its further endpoint $c_u$. One may view $\mathcal{G}$ as an infinite binary tree directed downwards.

We will use the notion of levels to define the metric on $\mathcal{G}$. Let $d_I\colon I\times I\to \mathbb{R}$ be a standard metric defined on an arc $I$. 
For each level $n\in\mathbb{N}$ of arcs in $\mathcal{G}$ we set the scale coefficient $\frac{1}{2^{2n}}$. 
We endow each arc $A_u$ with metric $d_I$ induced by a natural homeomorphism, multiplied by the scale coefficient, that is $\diam A_u=\frac{1}{2^{2|u|}}$, where $|u|$ is length of the indexing sequence $u$. Now, as we know that each pair of points in $\mathcal{G}$ is connected by a unique (and possibly infinite, when one of those points is an endpoint) path, we define the "taxicab" metric $d_{\mathcal{G}}\colon \mathcal{G}\times\mathcal{G}\to\mathbb{R}$ as a sum of lengths of arcs at consecutive levels composing that path. Moreover, we will say that a map $F\colon\mathcal{G}\to \mathcal{G}$ is linear on a subset $J\subset \mathcal{G}$ if there exists $a\in\mathbb{R}$ such that $d_\mathcal{G}(F(x), F(y))=a\cdot d_\mathcal{G}(x, y)$ for all $x,y\in J$. The map $F$ is (countably) piecewise linear if $\mathcal{G}=\bigcup_{J\in\mathcal{J}} J$, where $F$ is linear on each subset $J$. It is often the case that each component $J$ is an arc.

\begin{definition}
    We say that $J$ is a linearity interval for $F$ if $J$ is an arc contained in an edge $[c_\omega, c_{\omega\alpha}]$ for $\omega\in\{0,1\}^n$ and $\alpha\in\{0,1\}$, maximized in the sense of inclusion, on which the map $F$ is linear.
\end{definition}

To perform our main construction, we will also need a proper approximation of desired Cantor system $(\mathcal{C},f)$.
Following Corollary \ref{shimomura_fix}, we will construct a sequence 
of covers which will be guideline for successive approximation of $\mathcal{C}$ and $f$.
To achieve that, set $\mathcal{U}_0=\{\mathcal{C}\}$ and then for each $i\in\mathbb{N}_1$ let $\mathcal{W}_i$ be a decomposition further refining $f^{-1}(\mathcal{U}_{i-1})\vee \mathcal{U}_{i-1}$ with $\mesh(\mathcal{W}_i)< \frac{1}{i} \mesh(\mathcal{U}_{i-1})$. Define $\mathcal{U}_i$ as a refinement of $\mathcal{W}_i$, where each set $W\in\mathcal{W}_i$ is yet again divided into four distinct sets. For the sequence $(\mathcal{U}_i)_{i=0}^\infty$ define the inverse space $\mathrm{U}$ as in Corollary \ref{shimomura_fix}. By Theorem \ref{shimomura_tower}, the system $(V_{\mathrm{U}}, E_{\mathrm{U}})$ is conjugate to $(\mathcal{C}, f)$.

\par We can start transferring the structure of $\mathrm{U}$ onto the dendrite $\mathcal{G}$. Set the singleton graph to be represented by the root $c$ of $\mathcal{G}$,  which can be identified with the empty binary sequence. Let $n_i\in\mathbb{N}$ be defined inductively as $n_i=n_{i-1}+|\mathcal{U}_i|$ for each $i\in\mathbb{N}_1$ and $n_0=0$. Fix some $i$ and assume that each $U\in \mathcal{U}_i$ is identified with some binary code $\phi(U)\in \{0,1\}^{n_i}$. For each $V\in \mathcal{U}_{i+1}$ we assign a unique code $\omega_V\in \{0,1\}^{|\mathcal{U}_{i+1}|}$. Then we set the unique binary code $\phi(V)$ distinguishing $V$ as a concatenation $\phi(V):=(\phi(U),\omega_V)\in \{0,1\}^{n_{i+1}}$, where $U\in \mathcal{U}_i$ is the unique element such that $V\subset U$. Using the function $\phi$, we identify each set $U\in\mathcal{U}_i$ with a branching point $c_{\phi(U)}$ at level $n_i$.

Observe that we have introduced the structure of the inverse limit
by distinguishing a particular subset of branching points of $\mathcal{G}$ which corresponds to a particular subdendrite of $\mathcal{G}$. Simply, fix any $x\in \mathcal{C}$ and let $U_i=U_i^x\in \mathcal{U}_i$
be its representation by elements of clopen partitions, that is $\cap U_i=\{x\}$. Then the arcs $P_0=[c,c_{\phi(U_0)}]$ and $P_{i+1}=[c_{\phi(U_i)},c_{\phi(U_{i+1})}]$ for $i\geq 0$ intersect only at their endpoints, and thus $L_i=P_0\cup \ldots \cup P_i$ is a path between the root and the associated branching point at level $n_i$. Then $L_x=[c,\gend(x)]:=\overline{\bigcup_i L_i}$ is a unique path in $\mathcal{G}$ starting at $c$ and ending at an endpoint $\gend(x)$ of $\mathcal{G}$ induced by this path.

This provides a 1-1 correspondence between $\mathcal{C}$ and a subset of $\text{End}(\mathcal{G})$. Furthermore, the assignment $x\mapsto \gend(x)$
is continuous, because for every $\epsilon>0$ and any $\epsilon$-close points $x,y\in \mathcal{C}$ there exists $N=N(\epsilon)$ such that $U^x_i=U^y_i$ for $i\leq N$ and so the length of the arcs $L_x\setminus L_y$, $L_y\setminus L_x$, and consequently $d(x,y)$, is bounded by $\frac{1}{2^N}$. 
\begin{remark}
    This bound comes from the choice of the scale coefficient $\alpha_n=\frac{1}{2^{2n}}$ in the definition of the metric on $\mathcal{G}$. It implies that at each level $n$ in $\mathcal{G}$ the total length of all edges at this level is equal to $\frac{1}{2^n}$, and thus the distance $d(x,y)$ can be bounded by the length of all edges at levels $N+1$ and below, which is equal to $\frac{1}{2^N}$. One can bound $d(x,y)$ even more restrictively e.g. by the length of all paths descending from the endpoint $b_N$ of the path $[c, b_N]=L_x\cap L_y$, which is equal to $\frac{1}{2^{2N}}$.
\end{remark}

As a consequence, the set $E=\{\gend(x)\colon x\in \mathcal{C}\}$ is a closed subset of the set of endpoints of $\mathcal{G}$. Let $\mathcal{S}=\bigcup_{x\in \mathcal{C}}L_x$  be a subdendrite in $\mathcal{G}$ obtained by connecting elements of the closed set $E$ with the root $c$ and let $d_\mathcal{S}$ be the metric induced by $d_{\mathcal{G}}$. Observe that since for each $U\in \mathcal{U}_i$ there are at least $4$ elements $V_1, V_2, V_3, V_4$ in $\mathcal{U}_{i+1}$ with codes $\phi(V_i)$ extending $\phi(U)$, by results of \cite{Charatonik} we obtain that $\mathcal{S}$ is homeomorphic to $\mathcal{G}$, i.e. it is also a Gehman dendrite. Observe that, by the definition of $\mathcal{S}$, we have $\mathrm{End}(\mathcal{S})=E$. We will also use the following notation: for any $U\in \mathcal{U}_i$, by subdendrite of $\mathcal{G}$ with root $c_{\phi(U)}$ we mean the dendrite $$X_{c_{\phi(U)}}=\bigcup_{x\in U}L_x.$$

\subsection{The dynamics on $\mathcal{S}$}\label{F_map}
We will use the notation introduced in Section~\ref{sec:setup}. Start by putting $F(c)=c$. Next, for each $U\in\mathcal{U}_i$, define $F(c_{\phi(U)})=c_{\phi(V)}$, where $V\in\mathcal{U}_{i-1}$ is the unique set containing all sets $W\in\mathcal{U}_{i}$ such that $f(U)\cap W\neq \emptyset$.
 
\begin{remark}
The existence and uniqueness of the set $V$ comes from our construction 
since all bounding maps are +directional and  edge-surjective. Moreover, observe that if $U_i\in\mathcal{U}_i$ and $U_{i+1}\in\mathcal{U}_{i+1}$ is such that $U_{i+1}\subset U_i$, then for $F(c_{\phi(U_i)})=c_{\phi(V_{i-1})}$ and $F(c_{\phi(U_{i+1})})=c_{\phi(V_{i})}$ we have that the binary encoding of $V_{i-1}$ is a prefix of $V_{i}$, as each $\mathcal{U}_{i+1}$ is a family of subsets of $f^{-1}(\mathcal{U}_i)\vee \mathcal{U}_i$. Thus if $(U_i)_{i=0}^{\infty}\in V_{\mathrm{U}}$, then for $(V_i)_{i=0}^\infty$ obtained by translating back the action of $F$ on $(c_{\phi(U_i)})_{i=0}^\infty$ we have that $(V_i)_{i=0}^\infty\in V_\mathrm{U}$.
 \end{remark}

For each branching point $b\in\mathcal{S}$ outside levels $n_i$, i.e. any branching point $b$ which is contained in the path $P=[c_{\phi(U_{i-1})}, c_{\phi(U_{i})}]$ for some $U_{i-1}\in\mathcal{U}_{i-1}, U_i\in\mathcal{U}_i$, but is not endpoint of $P$, we define $F(b)=F(c_{\phi(U_{i-1})})$.

Fix any $U\in \mathcal{U}_{i}$ and let $V\in \mathcal{U}_{i-1}$ be the unique set such that $F(c_{\phi(U)})=c_{\phi(V)}$. Denote by $\mathcal{W}_{\text{ad}}(U)$ the set of all $W\in \mathcal{U}_i$ such that $\phi(V)$ is a prefix of $\phi(W)$.
Similarly, we define
$\mathcal{W}_{\text{rel}}(U)=\{W\in \mathcal{U}_{i} : f(U)\cap W\neq \emptyset\}$.  Remind that each $W\in\mathcal{U}_i$ is divided into at least four distinct sets from $\mathcal{U}_{i+1}$ and thus there exists at least one edge of the form $[c_{\phi(W)}, c_v]$, i.e. $v=(\phi(W),a)$ for some $a\in \{0,1\}$. For every $W$ pick one of such edges and denote it by $e_{W}$.

Fix any path $P=[c_{\phi(U_{i-1})}, c_{\phi(U_i)}]$ for some $U_{i-1}\in\mathcal{U}_{i-1}, U_i\in\mathcal{U}_i$
and fix any edge $e=[e_u,e_v]\subset P$.

Assume first that $v\neq c_{\phi(U_i)}$. Enumerate elements of the set $\mathcal{W}_{\text{ad}}(U_{i-1})=\{W_1,\ldots, W_k\}$. Divide the edge
$e$ onto $2k$ consecutive subintervals $I_{2j}, I_{2j+1}$ for $j\in\{1,...,k\}$. Recall that by the definition of $F$ both endpoints of 
$e$ are mapped onto the same point $F(c_{\phi({U_{i-1}})})$. Now for each $j=1,...,k$ perform a two-way linear stretch of $I_{2j},I_{2j+1}$ over the paths $P_{W_j}=[F(c_{\phi({U_{i-1}})}), c_{\phi(W_j)}]$, that is $F(I_{2j})=F(I_{2j+1})=P_{W_j}$, where on $I_{2j}$ we preserve orientation of $P_{W_j}$ and on $I_{2j+1}$ we revert it. Effectively, one can see that the image $F(e)$ covers the largest binary tree $X_{F(c_{\phi(U_{i-1})})}\cap \mathcal{S}_{n_i}$ contained in the first $n_i$ levels of $\mathcal{S}$ such that the point $F(c_{\phi(U_{i-1})})$ is its root.

Next, let us consider the case of $[b, c_{\phi(U)}]$ for some $U\in\mathcal{U}_1$. Enumerate elements of the set $\mathcal{W}_{\text{rel}}(U)=\{W_1,\ldots, W_k\}$ and recall that $F(b)=F(c_{\phi(U)})=c$. Divide the edge $e$ onto  $2k$ consecutive intervals $I_{2j}, I_{2j+1}$ for $j\in\{1,...,k\}$. Then for each $j=1,...,k$ perform a two-way linear stretch over the paths $P'_{W_j}=[c, c_{\phi(W_j)}]\cup e_{W_j}$, that is $F(I_{2j})=F(I_{2j+1})=P'_{W_j}$, where on $I_{2j}$ we preserve orientation of $P'_{W_j}$ and on $I_{2j+1}$ we revert it.

Finally, let $e\subset\mathcal{S}$ be an edge $[b, c_{\phi(U)}]$ for some $U\in\mathcal{U}_i$, where $i\geq 2$.
Like it has been done before, enumerate elements of the set $\mathcal{W}_{\text{rel}}(U)=\{W_1,\ldots, W_k\}$.
Divide the edge $e$ onto $2k+1$ consecutive subintervals $I_0\cup \{I_{2j}, I_{2j+1}\colon j=1,\ldots,k\}$. Recall that by the definition of $F$ there are sets $V_{i-2}\in\mathcal{U}_{i-2}$, $V_{i-1}\in\mathcal{U}_{i-1}$ such that $F(b)=c_{\phi(V_{i-2})}$ and $F(c_{\phi(U)})=c_{\phi(V_{i-1})}$. Set $F(I_0)$ to be a linear stretch over the path $P''=[c_{\phi(V_{i-2})}, c_{\phi(V_{i-1})}]$, preserving the orientation of $\mathcal{S}$. Then for each $j=1,...,k$ perform a two-way linear stretch over the paths $P''_{W_j}=[c_{\phi(V_{i-1})}, c_{\phi(W_j), 0}]\cup e_{W_j}$, that is $F(I_{2j})=F(I_{2j+1})=P''_{W_j}$, where on $I_{2j}$ we preserve orientation of $P''_{W_j}$ and on $I_{2j+1}$ we revert it.

Effectively, one can see that the image $F(e)$ in the last two cases covers not only every path leading to any set $W_\text{rel}$ related to $U$, but also a directly descending edge from that set, expanding down in $\mathcal{S}$ beyond the $n_i\text{-th}$ level.

Set the action on $\mathrm{End}(\mathcal{S})$ in the following way: let $x_{\text{end}}\in\mathrm{End}(\mathcal{S})$ and $(x_n)\subset \mathcal{S}\setminus \text{End}(\mathcal{S})$ be a sequence converging to $x_{\text{end}}$. Set $F(x_{\text{end}})=\lim_{n\to\infty}F(x_n)$, provided the limit exists. As we will see in next Lemma, it is always the case.

\begin{lemma}\label{F_cont}
The map $F\colon \mathcal{S}\to\mathcal{S}$ is well-defined and continuous.
\end{lemma}
\begin{proof}
    It is clear that $F$ is continuous on $\mathcal{S}\setminus\text{End}(\mathcal{S})$, as every edge is mapped piecewise linearly respecting the action defined on its endpoints.
    
    Moreover, one can see that the image of any sequence $(x_n)\subset \mathcal{S}\setminus \text{End}(\mathcal{S})$ converging to an endpoint $x_{\text{end}}$ is a Cauchy sequence. Observe that for each $(x_n)\to x_{\text{end}}$ and $U_i\in\mathcal{U}_i$, a neighborhood of a corresponding point in the Cantor set $\mathcal{C}(x_\text{end})\in U_i$, there is $N_i\in\mathbb{N}$ such that the sequence $(x_n)_{n=N_i}^\infty$ is contained entirely in a subdendrite $X_{c_{\phi(U_i)}}\subset\mathcal{S}$ defined by its root $c_{\phi(U_i)}$. One can see that the image $F(X_{c_{\phi(U_i)}})$ is thus contained entirely in $X_{F(c_{\phi(U_i)})}$ and $\diam(X_{F(c_{\phi(U_i)})})$ decreases exponentially as $i\to\infty$.
    
    Thus the natural extension of $F$ onto $\text{End}(\mathcal{S})$ given by $F(x_{\text{end}})=\lim_{n\to\infty}F(x_n)$ is well-defined.  Moreover, it is continuous, as yet again for any $\epsilon>0$ and any $x_{\text{end}}\in\text{End}(S)$ one can find its neighborhood $U_i\in\mathcal{U}_i$ such that $\diam(X_{F(c_{\phi(U_i)})})<\epsilon$.

\end{proof}

\begin{lemma}\label{conjugate_endpoint}
The 
subsystem $(\text{End}(\mathcal{S}), F|_{\text{End}(\mathcal{S})})$ is conjugate to $(\mathcal{C}, f)$.
\end{lemma}
\begin{proof}
Observe that we have the following three bijections: $f_1\colon\mathcal{C}\to V_\mathrm{U}$, which identifies every point $c$ with a unique inverse sequence $(U_i)_{i=1}^\infty$ such that $\bigcap_{i=1}^\infty U_i=c$, $f_2\colon V_\mathrm{U}\to\mathcal{G}$, which identifies these inverse sequences $(U_i)_{i=1}^\infty$ with sequences of branching points $(c_{\phi(U_i)})_{i=1}^\infty$ and $f_3\colon \mathcal{G}\to \mathrm{End}(\mathcal{G})$ identifying these sequences with their limits $\lim_{i\to\infty} c_{\phi(U_i)}=\mathrm{end}(c)$.

We claim that the two systems are conjugated via a composition of these three maps. It is easy to see that the maps $f_1, f_3$ are conjugacies. It remains to show that the map $f_2$ is indeed a conjugacy map.

To see that, remind that the map $F\colon\mathcal{G}\to\mathcal{G}$ puts $F(c_{\phi(U)})=c_{\phi(V)}$ if and only if for the set $U\in\mathcal{U}_i$ there exists $W\in \mathcal{U}_i$ such that $f(U)\cap W\neq \emptyset$ and $W\subset V\in\mathcal{U}_{i-1}$. Thus, by the definition of the inverse space $\mathrm{U}$, if for two inverse sequences $(U_i)_{i=1}^\infty, (V_i)_{i=1}^\infty\in V_\mathrm{U}$ we have that $f(U_i)\cap V_i\neq\emptyset$, then $F(c_{\phi(U_i)})=c_{\phi(V_{i-1})}$ for $i\in\mathbb{N}$ and consequently $F(c_{\phi(U_{i})})_{i=1}^\infty=(c_{\phi(V_i)})_{i=1}^\infty$.

On the other hand, if we have that $F(c_{\phi(U_{i})})_{i=1}^\infty=(c_{\phi(V_i)})_{i=1}^\infty$, then $F(c_{\phi(U_i)})=c_{\phi(V_{i-1})}$. Consequently, in each $V_{i-1}\in\mathcal{U}_{i-1}$ there exists a set $W_i\in\mathcal{U}_i$ such that $f(U_i)\cap W_i\neq\emptyset$. Assume that $V_i\neq W_i$ for some $i\in\mathbb{N}$. Then, by the definition of the inverse space $\mathrm{U}$, $V_k\neq W_k$ for all $k\geq i$ and thus the sequence $(V_i)_{i=1}^\infty\neq (W_i)_{i=1}^\infty$. Since we have that $F(c_{\phi(U_k)})=c_{\phi(V_{k-1})}$, then by the definition of the inverse space $\mathrm{U}$ we have that $f(U_k)\subset V_{k-1}$. Thus, $f(U_k)\cap W_k\neq\emptyset$, $f(U_k)\subset V_{k-1}$ and $V_{k-1}\neq W_{k-1}$ for all $k>i$, a contradiction.
\end{proof}

\begin{lemma}\label{eventually_cover}
    The map $F$ is pure mixing.
\end{lemma}
\begin{proof}
    To see that, fix some open interval $I$ contained in an edge of $\mathcal{S}$.
    One can see that $F$ stretches every edge piecewise linearly with each linear stretch having the stretching coefficient $\lambda >\frac{5}{4}$. This comes from the following combination of facts:
    \begin{enumerate}
        \item the definition of the metric $d_\mathcal{S}$ and the $\alpha$ coefficients,
        \item the construction of the sequence of refinements $(\mathcal{U}_i)$,
        \item the definition of the action of $F$ on the edges.
    \end{enumerate}
    First, observe that the image of both edges from the first level covers multiple paths composed of edges from at least levels 1 and 2, and the lengths of those paths is at least $1+\frac{1}{4}$ of the length of any edge from level 1. Then observe that the image of every other edge covers multiple times an edge located upper in the dendrite and the ratio between lengths of edges from two consecutive levels is equal to $4:1$.

    Let $J_1, J_2$ be two adjacent linearity intervals for $F$. If we have that $I\cap J_1\neq\emptyset, I\cap J_2\neq \emptyset$, then $J_1
    \cap J_2=z\in I$ and $F(z)$ is a branching point. If $I\subset J_1$, then by the fact that $J_1$ is stretched linearly with $\lambda > 1$, there exists $k\in\mathbb{N}$ such that $F^k(I)$ will contain a common endpoint of two linearity intervals. Thus there exists $k_1\in\mathbb{N}$ such that $F^{k_1}(I)$ will contain a branching point $b$. By the definition of $F$, the image $F(b)$ is a branching point $c_{\phi(U)}$ for some $U\in\mathcal{U}_i$ and thus $F^{i+1}(b)=c$. If the image of $I$ after $k_1+i+1$ iterations of $F$ does not contain a whole edge from level 1, then by the fact that $F(c)=c$ and $\lambda>1$ there exists $k_2\in\mathbb{N}$ such that for $F^{k_1+k_2+i+1}(I)$ this is true. Without loss of generality assume that the image $F^{k_1+k_2+i+1}(I)$ contains the edge $e_0=[c, c_0]$.

    Now observe that the image $F(e_0)$ covers every edge from levels 1 to $n_1$, where $n_1$ was defined to be the cardinality of $\mathcal{U}_1$. Thus the second iteration $F^2(e_0)$ will cover in addition at least one edge descending directly from each branching point $c_{\phi(U_1)}$ for every $U_1\in\mathcal{U}_1$. These edges under the action of $F$ exhibit a similar behavior as $e_0$, resulting with $F^3(e_0)$ covering every edge from levels 1 to $n_1+n_2$. Now in the fourth iteration the image $F^4(e_0)$ covers at least one edge descending directly from each branching point of the form $c_{\phi(U_2)}$, which results with $F^5(e_0)$ covering every edge from levels 1 to $n_1+n_2+n_3$. This ``flooding''` continues indefinitely, eventually covering every edge in $\mathcal{S}$.

    The immediate result is that $F$ is a mixing transformation of $\mathcal{S}$. On the other hand one can see that for any $n\in\mathbb{N}$ we have $F^n(e_0)\cap \text{End}(\mathcal{S})=\emptyset$, completing the proof.
\end{proof}

\begin{proof}[\textbf{Proof of Theorem A}]
Fix a Cantor dynamical system $(\mathcal{C},f)$.
Let $\mathcal{S}$ be the Gehman dendrite induced by $\mathcal{C}$ as presented in Section \ref{sec:setup} and let $F$ be the map induced on $\mathcal{S}$ by 
$(\mathcal{C},f)$. By Lemma \ref{F_cont} $F$ is continuous and by the construction $\text{End}(\mathcal{S})$ is an invariant subset in the dynamical system $(\mathcal{S},F)$.

By Lemma \ref{conjugate_endpoint}
$(\text{End}(\mathcal{S}), F|_{\text{End}(\mathcal{S})})$ is conjugate to the initially chosen Cantor system $(\mathcal{C}, f)$. 
Finally, by Lemma \ref{eventually_cover}, dynamical system $(\mathcal{S}, F)$ is pure mixing. The proof is complete.
\end{proof}

\section{Proof of Theorem B}
The map $F\colon \mathcal{S} \to\mathcal{S}$ constructed in Section \ref{F_map} can be modified to an exact map of the Gehman dendrite. The change we will introduce is based on the modification of the action on the first level of edges in $\mathcal{S}$. To achieve that, we will use the following method to derive surjective maps from more general mappings onto the space of subsets of the dendrite. 
\begin{definition}
    Let $X, Y$ be metric spaces. We say that $F\colon X\to 2^Y$ is an upper semicontinuous function if for every $x\in X$ and open $V\subset Y$ such that $F(x)\subset V$ there exists an open neighborhood $x\in U\subset X$ such that $F(U)\subset V$.
\end{definition}
\begin{theorem}[\cite{Nadler}, Theorem 7.4]\label{GMT}
Let $X, Y$ be compact metric spaces and $F_n\colon X\to 2^Y$ be a countable sequence of upper semicontinuous functions satisfying the following conditions:
\begin{enumerate}
    \item $F^{n+1}(x)\subset F_{n}(x)$ for all $n\in\mathbb{N}$ and $x\in X$,
    \item $\lim_{n\to\infty}\diam(F_{n}(x))=0$ for all $x\in X$,
    \item $F_n(X)=Y$ for all $n\in\mathbb{N}$.
\end{enumerate}
Then the map $f\colon X\to Y$ defined by $f(x)=\bigcap_{n\in\mathbb{N}} F_{n}(x)$ for every $x\in X$ is continuous and onto.
\end{theorem}

\subsection{Surjective mapping of the first level of edges}\label{f_map}
We will define inductively a sequence of upper semicontinuous functions $F_n\colon [c, c_0]\cup[c, c_1]\to 2^\mathcal{S}$ and derive $f$ from that sequence using Theorem \ref{GMT}. The action $f$ will be defined analogously on both connected components of $\mathcal{S}\setminus\{c\}$. Thus, without the loss of generality, we will focus on the edge $[c, c_0]$. All of the presented reasoning below applies directly to the case $[c, c_1]$.

\begin{itemize}
\item $n=0$. Set $F_0(x)=\mathcal{S}$ for each $x\in[c, c_0]$.

\item $n=1$. Let us start by dividing the edge $[c, c_0]$ into eight subintervals of equal length, meaning that we set cut points $k_1, k_2, k_3$ in $\frac{1}{8}, \frac{1}{4}, \frac{3}{8}$ of the length of $[c, c_0]$, a cut point $i$ in $\frac{1}{2}$ of the length of $[c, c_0]$ and cut points $j_1, j_2, j_3$ in $\frac{5}{8}, \frac{3}{4}, \frac{7}{8}$ of the length of $[c, c_0]$. Set $F_1(x)=X_{c_0}$ for each $x\in[k_1, k_3]$ and $F_1(x)=X_{c_1}$ for each $x\in[j_1, j_3]$. Then on the  rest of $[c,c_0]$ we define $F_1$ as a linear map in the following way. We stretch linearly $F_1[c, k_1)=[c, c_0)=F_1(k_3, i]$ and $F_1[i, j_1)=[c, c_1)=F_1(j_3, c_0]$, preserving the orientation on $[c, k_1), [i, j_1)$ and reversing it on $(k_3, i], (j_3, c_0]$.
\end{itemize}
From now on, we will modify the map $F_{n}$ on those intervals, which have been transformed into dendrites, leaving those stretched over edges intact.
\begin{itemize}
\item $n+1$. Let $I=[i_1, i_2]$ be an interval (maximized in terms of inclusion) such that $F_n(I)=X_{c_{\omega_1...\omega_n}}$ for some branching point $c_{\omega_1...\omega_n}\in\mathcal{S}$. We also assume that there is $\varepsilon>0$ such that $F_n$ on $[i_1-\varepsilon,i_1)$ and $(i_2,i_2+\varepsilon]$ is a linear map, and $\lim_{x\to i_1^+}F_n(x)=\lim_{x\to i_1^-}F_n(x)=c_{\omega_1...\omega_n}$.  In short, assume that the modification from $n$-th step has been completed.

Let us start by dividing $I$ into eight subintervals of equal length, setting cutpoints $l_1, l_2, l_3, j, m_1, m_2, m_3$ in the same manner as we did before. If both edges $[c_{\omega_1...\omega_n}, c_{\omega_1...\omega_n 0}], [c_{\omega_1...\omega_n}, c_{\omega_1...\omega_n 1}]$ are in $\mathcal{S}$, set $F_{n+1}(x)=X_{c_{\omega_1...\omega_n 0}}$ for $x\in [l_1, l_3]$ and $F_{n+1}(x)=X_{c_{\omega_1...\omega_n 1}}$ for $x\in [m_1, m_3]$. Then stretch linearly $F_{n+1}[i_1, l_1)=[c_{\omega_1...\omega_n}, c_{\omega_1...\omega_n 0})=F_{n+1}(l_3, j]$ and $F_{n+1}[j, m_1)=[c_{\omega_1...\omega_n}, c_{\omega_1...\omega_n 1})=F_{n+1}(m_3, i_2]$. If $[c_{\omega_1...\omega_n}, c_{\omega_1...\omega_n a}]$ is the only one edge attached to $c_{\omega_1...\omega_n}$, we set the image for both subintervals in the same way, but over the same dendrite $X_{c_{\omega_1...\omega_n a}}\cup[c_{\omega_1...\omega_n}, c_{\omega_1...\omega_n a}]$.
\end{itemize}

\begin{lemma}
For each $n\in\mathbb{N}$ the function $F_n$ is upper semicontinuous.
\end{lemma}
\begin{proof}
It is easy to see that the linear stretches of $F_n$ are continuous. If $x$ is in the interior of a subinterval of $[c,c_0]$ which is mapped by $F_n$ into a subdendrite $X_{\text{sub}}$ in $\mathcal{S}$, then one can find a neighborhood $x\in U\subset [c,c_0]$ such that $F_n(U)=X_{\text{sub}}$. Assume that $x$ is on the boundary of such subinterval. Observe that, as we defined the stretch to either preserve or reverse the orientation in $\mathcal{S}$, one could say that $F_n$ resemble continuity on $\mathcal{S}$. Thus for each open $V\subset \mathcal{S}$ one can find a neighborhood $x\in U\subset [c, c_0]$ such that its image will not stretch outside $V$.
\end{proof}

\begin{lemma}
    The sequence $(F_n)_{n=0}^\infty$ satisfies the conditions from Theorem \ref{GMT}.
\end{lemma}

\begin{proof}
Let $x\in [c,c_0]$.
\begin{enumerate}
    \item Consider the sequence of images $(F_n)(x)_{n=0}^{\infty}$. It is clear that $F_0(x)=\mathcal{S}$. Observe that there are two cases:
    \begin{enumerate}
        \item there exists $n\in\mathbb{N}$ such that $F_n(x)=y\in\mathcal{S}$, i.e. the sequence stabilises,
        \item for each $n\in\mathbb{N}$ the image $F_n(x)$ is a subdendrite in $\mathcal{S}$.
    \end{enumerate}
    Thus $(F_n)(x)_{n=0}^{\infty}$ is a sequence of subdendrites in $\mathcal{S}$ which either become trivial or not. Assume that $F_n(x), F_{n+1}(x)$ are two consecutive nontrivial dendrites. Since in $F_{n+1}$ we modify the action on intervals put previously by $F_n$ into a dendrite $X_{c_{\omega_1...\omega_n}}$ by either distributing it over the edge $[c_{\omega_1...\omega_n}, c_{\omega_1...\omega_{n+1}}]$ or putting it into the dendrite $X_{c_{\omega_1...\omega_{n+1}}}$, we have that $F_n(x)=X_{c_{\omega_1...\omega_n}}, F_{n+1}(x)=X_{c_{\omega_1...\omega_{n+1}}}$ and consequently $F_{n+1}(x)\subset F_n(x)$. Thus we have
    \begin{enumerate}
        \item $F_0(x)\supset F_1(x)\supset...\supset  F_n(x)= F_{n+1}(x)=...$ when the sequence stabilises,
       \item $F_0(x)\supset F_1(x)\supset F_2(x)\supset...$ when it does not stabilise.
    \end{enumerate}
    In both cases one can see that condition 1 from Theorem \ref{GMT} is satisfied.
    \item Let us continue the analysis from the previous paragraph. Observe that whenever $X_{c_{\omega_1...\omega_{n+1}}}\subset X_{c_{\omega_1...\omega_n}}$, then we have that $\diam(X_{c_{\omega_1...\omega_{n+1}}})\leq \frac{1}{4}\diam(X_{c_{\omega_1...\omega_n}})$. Thus
    \begin{enumerate}
        \item if the sequence $(F_n)(x)$ stabilises, then $$\diam(F_n(x))=\diam(F_{n+1}(x))=...=0,$$
        \item if the sequence does not stabilise, then $$\diam F_{n+1}(x)\leq \frac{1}{4}\diam(F_{n}(x))\leq (\frac{1}{4})^{n+1} F_0(x)\leq(\frac{1}{4})^{n+1}.$$
    \end{enumerate} In both cases one can see that condition 2 from Theorem \ref{GMT} is satisfied.
    \item Observe that for each $n\in\mathbb{N}$ the map $F_n$ covers piecewise linearly the first $n$ levels in $\mathcal{S}$ and then assigns every subdendrite of the form $X_{c_{\omega_1...
    \omega_n}}$ to an interval $I_{\omega_1...\omega_n}\subset [c, c_0]\cup [c, c_1]$. Thus, when considering the total image of $F_n$, one can see that $\bigcup_{x\in[c, c_0]\cup[c, c_1]} F_n(x)=\mathcal{S}$. Thus, condition 3 from Theorem \ref{GMT} is satisfied.
\end{enumerate}
\end{proof}

Let $f\colon [c,c_0]\cup [c,c_1]\to \mathcal{S}$ be fedined by $f(x)=\bigcap_{n\in\mathbb{N}} F_n(x)$ for all $x\in [c,c_0]\cup [c,c_1]$. By Theorem \ref{GMT}, $f$ is continuous and onto. Moreover, $f(c)=f(c_0)=f(c_1)=c$.\newline

\subsection{Exact map on $\mathcal{S}$}\label{modify_F}
Now we modify the map $F$ from Section \ref{F_map}. We simply replace the action on $[c, c_0]\cup[c, c_1]$ in the original map $F\colon \mathcal{S}\to\mathcal{S}$ with the function $f\colon [c,c_0]\cup [c,c_1]\to \mathcal{S}$.

\begin{corollary}\label{F_mod_cont}
    The modified action $F_{\text{mod}}\colon\mathcal{S}\to\mathcal{S}$ is continuous on $\mathcal{S}$ and the endpoint subsystem $(\text{End}(\mathcal{S}), F_{\text{mod}}|_{\text{End}(\mathcal{S})})$ is conjugated to the initially chosen Cantor system.
\end{corollary}
\begin{proof}
    The continuity follows from the fact that both $F$ and $f$ are continuous by Lemma \ref{F_cont} and Theorem \ref{GMT} respectively, and both actions agree on $c, c_0, c_1$. The conjugacy follows from Lemma \ref{conjugate_endpoint}, as the modification does not affect the action on $\mathrm{End}(\mathcal{S})$.
\end{proof}
\begin{lemma}\label{F_mod_exact}
    The action $F_{\text{mod}}\colon\mathcal{S}\to\mathcal{S}$ is exact.
\end{lemma}
\begin{proof}
    We will follow the reasoning presented in the proof of Lemma \ref{eventually_cover}. Before that, 
    let us start with an observation on the modified part of the original map $F\colon\mathcal{S}\to\mathcal{S}$. Namely, observe that the first level of edges is piecewise linearly stretched on $\mathcal{S}$ and the stretching coefficient of those stretches on each linearity interval is the same, $\lambda=8$. The main difference for top level, however, is that it contains countably many intervals of linearity in it, and there is a Cantor set in the complement of their interiors in $[c,c_0]$, resp. $[c,c_1]$. Let us explain the argument behind this claim in more detail.
    \par The characteristics of $\mathcal{S}$ holds in the same way as in the proof of Lemma \ref{eventually_cover}, so yet again one can see that the ratio between lengths of edges from two consecutive levels is $4:1$. Then inductively we can show the above assertion. Firstly observe that, since we divided the edge into eight intervals of equal length in the first step of the induction in the construction in Section \ref{f_map}, the first four stretches are those with intervals of $\frac{1}{8}$ of the original length of the edge. Then, as we modified the map inductively, observe that in each step we replicated the construction done in the first step of the induction on two intervals, which are both $\frac{1}{4}$ of the length of the interval modified in the previous step of the induction. Thus we perform similar stretches with the same coefficient $\lambda=8$, since the ratio between consecutive edges and the ratio between consecutive modification intervals are the same. Since the construction is insensitive on whether there are two edges attached to any branching point $c_{\omega}\in\mathcal{S}$ or not, as we duplicate covering if it is not the case, every linearity interval is evenly distributed on its image in $\mathcal{S}$. 
    \par An immediate consequence of the above reasoning is that $F_\text{mod}$ is a piecewise linear transformation of $\mathcal{S}$ with each linear stretch having the stretching coefficient $\lambda>4$. This follows from the fact that the original map $F$ had these stretching coefficients bounded by $\frac{5}{4}$ for linearity intervals from the first level in $\mathcal{S}$, while every other linearity interval admitted a stretch with $\lambda>4$.
    \par Fix some open interval $I\subset\mathcal{S}$. One can see that $F_\text{mod}^{-1}(\mathrm{End}(\mathcal{S}))$ forms a Cantor set in $\mathcal{S}$ and thus $I$ is not a subset of $F_\text{mod}^{-1}(\mathrm{End}(\mathcal{S}))$.
    If $I$ is contained entirely in a single linearity interval $J$, then by the fact that $J$ is stretched with $\lambda>4$ there exists $k\in\mathbb{N}$ such that $F^k_\text{mod}(I)$ will contain a common endpoint of two linearity intervals. If $I\cap J_1, I\cap J_2\neq\emptyset$ for two adjacent linearity intervals then clearly $I$ contains a common endpoint $z$ of $J_1, J_2$ and $F_\text{mod}(z)$ is a branching point in $\mathcal{S}$.
    \par Thus, there exists $k_1\in\mathbb{N}$ such that $F^{k_1}_\text{mod}(I)$ will contain a branching point $b$. Moreover, $F_\text{mod}(b)$ is a branching point $c_{\phi(U)}$ for some $U\in\mathcal{U}_i$ and thus $F_\text{mod}^{i+1}(b)=c$. If the image of $I$ after $k_1+i+1$ iterations does not contain a whole edge from level $1$, then by the fact that $F_\text{mod}(c)=c$ and both intervals of linearity adjacent to $c$ have length $1/8$ and $\lambda=8$, we see that 
    there exists $k_2\in\mathbb{N}$ such that for $F^{k_1+k_2+i+1}_\text{mod}(I)$ this is true.
    \par Then, since $F_\text{mod}([c, c_0])=F_\text{mod}([c, c_1])=\mathcal{S}$, we have that $F_\text{mod}^{k_1+k_2+i+2}(I)=\mathcal{S}$, completing the proof.

\end{proof}
\begin{proof}[\textbf{Proof of Theorem B}]
Fix a Cantor dynamical system $(\mathcal{C},f)$. Let $\mathcal{S}$ be the Gehman dendrite induced by $\mathcal{C}$ as presented in Section \ref{sec:setup}. Let $F$ be 
 the map induced on $\mathcal{S}$ by $(\mathcal{C},f)$ defined as in Section \ref{F_map} and $F_\text{mod}$ be obtained by 
 the modification 
presented in Section \ref{modify_F}.

The map $G$ is continuous and $\text{End}(\mathcal{S})$ is an invariant subset in the dynamical system $(\mathcal{S},F_\text{mod})$, which forms a endpoint subsystem
$(\text{End}(\mathcal{G}), F_\text{mod}|_{\text{End}(\mathcal{S})})$ conjugate to the initially chosen Cantor system $(\mathcal{C}, f)$, as shown in Corollary \ref{F_mod_cont}. 
Finally, by Lemma \ref{F_mod_exact}, dynamical system $(\mathcal{S}, F_\text{mod})$ is exact. The proof is complete.
\end{proof}

\section*{Acknowledgements}

The authors are grateful to Jozef Bobok, Klara Karasová, Benjamin Vejnar and Dominik Kwietniak for numerous conversations and helpful comments throughout the work on this paper. The work was partially supported by the project AIMet4AI No. CZ.02.1.01/0.0/0.0/17\_049/0008414.

\bibliographystyle{amsplain}

\begin{thebibliography}{10}

\bibitem{Akin} E. Akin, E. Glasner and B. Weiss, Generically there is but one self homeomorphism of the Cantor set, Trans. Amer. Math. Soc. {\bf 360} (2008), no.~7, 3613--3630.

\bibitem{monografia} E. Akin, M.~G. Hurley and J.~A. Kennedy, Dynamics of topologically generic homeomorphisms, Mem. Amer. Math. Soc. {\bf 164} (2003), no.~783.

\bibitem{Charatonik} Daniel Arévalo, Włodzimierz Charatonik, Patricia Pellicer Covarrubias, Likin Simón,
\textit{Dendrites with a closed set of end points},
Topology and its Applications,
Volume 115, Issue 1,
2001,
Pages 1-17.

\bibitem{Bald01} Stewart Baldwin,  \textit{Entropy estimates for transitive maps on trees.} Topology 40 (2001), no. 3, 551--569.

\bibitem{Bald07} Stewart Baldwin, \textit{Continuous itinerary functions and dendrite maps.} Topology Appl. 154 (2007), no. 16, 2889--2938.

\bibitem{Darji} N.~C. Bernardes Jr. and U.~B. Darji, Graph theoretic structure of maps of the Cantor space, Adv. Math. {\bf 231} (2012), no.~3-4, 1655--1680.

\bibitem{Boyle} M.~B.~M. Boyle and T. Downarowicz, The entropy theory of symbolic extensions, Invent. Math. {\bf 156} (2004), no.~1, 119--161.

\bibitem{Downarowicz} T. Downarowicz and S.~E. Newhouse, Symbolic extensions and smooth dynamical systems, Invent. Math. {\bf 160} (2005), no.~3, 453--499.

\bibitem{Gambaudo} Jean-Marc Gambaudo, Marco Martens, \textit{Algebraic Topology for Minimal Cantor Sets}, Ann. Henri Poincaré 7, 423–446 (2006).

\bibitem{Gehman} Harry Merrill Gehman, \textit{Concerning the subsets of a plane continuous curve}, Annals of Mathematics, Second Series, Volume 27, Number 1, September 1925, pp. 29-46.

\bibitem{Weiss} E. Glasner and B. Weiss, Quasi-factors of zero-entropy systems, J. Amer. Math. Soc. {\bf 8} (1995), no.~3, 665--686.

\bibitem{Glasner} E. Glasner and B. Weiss, The topological Rohlin property and topological entropy, Amer. J. Math. {\bf 123} (2001), no.~6, 1055--1070.

\bibitem{Dominik_drzewa} Grzegorz Harańczyk, Dominik Kwietniak, Piotr Oprocha, \textit{Topological structure and entropy of mixing graph maps}, Ergodic Theory and Dynamical Systems, Volume 34, Issue 05, October 2014, pp. 1587-1614.

\bibitem{Klara} Klara Karasová, Benjamin Vejnar. \textit{Topological fractals revisited}, https://arxiv.org/abs/2209.15394, (2022).

\bibitem{Kechris} A.~S. Kechris and C. Rosendal, Turbulence, amalgamation, and generic automorphisms of homogeneous structures, Proc. Lond. Math. Soc. (3) {\bf 94} (2007), no.~2, 302--350.

\bibitem{Kocan} Zdenek Kocan,  Veronika Kornecka-Kurkova, Michal Malek, \textit{Entropy, horseshoes and homoclinic trajectories on trees, graphs and dendrites}, Ergodic Theory and Dynamical Systems, 31(1), 165-175. 

\bibitem{exact_devaney} Dominik Kwietniak, Michał Misiurewicz, \textit{Exact Devaney chaos and entropy}, Qual. Theory Dyn. Syst. 6 (2005), no.~1, 169-179.

\bibitem{Kupka} Jiří Kupka, Piotr Oprocha, \textit{On the dynamics of generic maps on the Cantor set}, Topology Appl. 6 (2019), 330-342.

\bibitem{Kurka} P. Kurka, {\it Topological and symbolic dynamics}, Cours Sp\'ecialis\'es, 11, Soc. Math. France, Paris, 2003.

\bibitem{Li} Jian Li, Piotr Oprocha, Guohua Zhang, \textit{Quasi-graphs, zero entropy and measures with discrete spectrum}, Nonlinearity 35 (2022) 1360–1379.

\bibitem{Nadler} S. B. Nadler Jr., \textit{Continuum Theory: An Introduction}, Marcel Dekker, Inc., New York, 1992.

\bibitem{Nikiel} Jacek Nikiel, \textit{A characterizations of dendroids with uncountably many end points in the classical sense}, Houston J. Math. 9, (1983), 421-432.

\bibitem{Shimomura} Takashi Shimomura, \textit{Special homeomorphisms and approximation for Cantor systems},
Topology and its Applications,
Volume 161,
2014,
Pages 178-195.

\bibitem{Spi} 
Vladimír \v{S}pitalsk\'y,  \textit{Topological entropy of transitive dendrite maps.} Ergodic Theory Dynam. Systems 35 (2015), no. 4, 1289--1314.

\bibitem{Spi2} Vladimír \v{S}pitalsk\'y,  \textit{Entropy and exact Devaney chaos on totally regular continu}a. Discrete Contin. Dyn. Syst. 33 (2013), no. 7, 3135--3152.

\end{thebibliography}

\end{document}